\documentclass[12pt]{amsart}

\usepackage{amssymb}
\usepackage{enumerate}
\usepackage{graphicx}

\makeatletter
\@namedef{subjclassname@2010}{%
  \textup{2010} Mathematics Subject Classification}
\makeatother

\newtheorem{theorem}{Theorem}

\newtheorem{lemma}{Lemma}

\theoremstyle{definition}
\newtheorem{remark}{Remark}

\newcommand{\bR}{\mathbb{R}}

\newcommand{\ol}{\overline}
\newcommand{\ul}{\underline}

\newcommand{\Int}{{\textstyle \int}}

\newcommand{\diam}{\operatorname{diam}}

\numberwithin{equation}{section}

\begin{document}

\title[On the differentiation of integrals with respect to  translation invariant convex density   bases]
    {On the differentiation of integrals with respect to  translation invariant convex density  bases}

\author{Giorgi Oniani}

\address{Department of Mathematics, Akaki Tsereteli State University, 59, Tamar Mepe St., Kutaisi 4600, Georgia; e-mail: {\rm oniani@atsu.edu.ge}}

\maketitle

\begin{abstract} For a translation invariant convex density basis $B$ it is shown that its Busemann-Feller extension $B_{\mathrm{BF}}$ has close to $B$ properties, in particular, $B_{\mathrm{BF}}$ differentiates the same class of non-negative functions as $B$. Using the similarity between properties of the bases $B$ and $B_{\mathrm{BF}}$ some results known for Busemann-Feller  bases are transferred to bases without restriction  of being Busemann-Feller.

\medskip
\noindent \textbf{Key words:} Differentiation of integrals, translation  invariant basis,    Busemann-Feller basis, maximal operator.

\smallskip
\noindent \textbf{2010 Mathematics Subject Classification}: 28A15
\end{abstract}

\medskip
\medskip

\textbf{1. Definitions and notation}. A mapping $B$ defined on $\mathbb{R}^n$ is called a \emph{differentiation basis} (briefly: \emph{basis}) if for each $x\in \mathbb{R}^n$ the value  $B(x)$ is a family of bounded measurable sets of positive measure which contain $x$ and  such that there exists a sequence $(R_k)$ of sets from $B(x)$ with $\lim_{k\rightarrow\infty}\diam R_k = 0$.

Let $B$ be a basis. For $f\in L(\mathbb{R}^n)$  and $x\in \mathbb{R}^n$, the upper and lower limits of the integral means $\frac{1}{|R|}\int_R f$, where $R$ is an arbitrary set  from $B(x)$ and $\diam R \rightarrow0$, are called the \emph{upper and the lower derivatives with respect to $B$ of the integral of $f$ at the point $x$}, and are denoted by $\overline{D}_B(\int f,x)$ and $\underline{D}_B(\int f,x)$, respectively. If the upper and the lower derivatives coincide, then their combined value is called the \emph{derivative of $\int f$ at the point $x$} and denoted by $D_B(\int f,x)$. We say that \emph{$B$ differentiates }$\int f$ (or $\int f$ \emph{is differentiable with respect to} $B$) if $\ol{D}_B(\int f,x)=\ul{D}{\,}_B(\int f,x)=f(x)$ for almost all $x\in \bR^n$. If this is true for each $f$ in the class of functions $F\subset L(\mathbb{R}^n)$ we say that $B$ \emph{differentiates} $F$.

A basis $B$ is called:
\begin{enumerate}
\item[$\bullet$] \emph{homothecy invariant} if for every $x\in\mathbb{R}^n$, $R\in B(x)$  and a homothecy $H: \mathbb{R}^n \rightarrow \mathbb{R}^n$  we have that $H(R)\in B(H(x))$;

\item[$\bullet$] \emph{translation invariant} if for every $x\in\mathbb{R}^n$, $R\in B(x)$  and a translation $T: \mathbb{R}^n \rightarrow \mathbb{R}^n$  we have that $T(R)\in B(T(x))$;

\item[$\bullet$]  \emph{formed by sets from a class} $\Delta$ if  for every $x\in \mathbb{R}^n$ and $R\in B(x)$ we have that $R\in \Delta$;

\item[$\bullet$]  \emph{convex} if  it is formed by the class of convex sets;

\item[$\bullet$]  \emph{Busemann-Feller} if $(x\in\mathbb{R}^n, R\in B(x), y\in R)\Rightarrow R\in B(y)$;

\item[$\bullet$]  \emph{sub-basis of a basis} $B'$ (notation: $B\subset B' $)  if  $B(x)\subset B'(x)$ for every $x\in\mathbb{R}^n$;

\item[$\bullet$]  \emph{density} basis if  $B$ differentiates $\Int \chi_E$ for every bounded measurable set $E\subset \mathbb{R}^n$.
\end{enumerate}

Note that each homothecy invariant basis is translation invariant also.

In what follows the dimension of the space $\mathbb{R}^n$ is assumed to be greater
than $1$.

Denote by $\textbf{I}=\textbf{I}(\mathbb{R}^n)$ the basis for which $\textbf{I}(x)$ consists of all $n$-dimensional intervals containing the point $x$. Differentiation with respect to $\textbf{I}$ is called the \emph{strong differentiation}. Note that $\textbf{I}$ is a density basis, moreover, by virtue of the well-known result of Jessen, Marcinkiewicz and Zygmund (see, e.g. [1, p.50]) $\textbf{I}$ differentiates the class $L(1+\ln^{+}L)^{n-1}(\mathbb{R}^n)$.

The \emph{maximal operator $M_B$ and  truncated maximal operator $M_B^{r}$ $(r\in (0,\infty])$ corresponding to a basis} $B$ are defined as follows
\begin{align*}
M_B(f)(x)& =\sup_{R\in B(x)}\frac{1}{|R|}\int_R |f|,
\\
M_B^r(f)(x)& =\sup_{R\in B(x), \diam R<r}\frac{1}{|R|}\int_R |f|,
\end{align*}
where $f \in L_{\text{loc}}(\mathbb{R}^n)$ and $x\in\mathbb{R}^n$. Obviously, $M_B=M_B^\infty$.

Let us say that a basis $B$ is a \emph{measurable} if  for  any $f \in L(\mathbb{R}^n)$ the functions $\overline{D}_B\left(\Int f, \cdot\right)$, $\underline{D}_B\left(\Int f, \cdot\right)$  and $M_B^r(f)$ $(r\in (0,\infty])$ are measurable. It is easy to check that if $B$ is a translation invariant or Busemann-Feller basis then $B$ is measurable.

For measurable bases  $B$ and $B'$ let us introduce the following definitions:
\begin{enumerate}
\item[$\bullet$] We say that \emph{ $B'$ locally majorizes } $B$ (written as $B \leq B'$) if there exist $c\geq 1$ and $\delta>0$, such that the estimate
$$
|\{M_{B}^r(f)\geq\lambda\}|\leq c |\{M_{B'}^{cr}(f)\geq \lambda/c\}|\eqno(1)
$$
is fulfilled for every $f\in L(\mathbb{R}^n), r\in (0,\delta)$ and $\lambda\in (0,\infty)$;

\item[$\bullet$] We say that \emph{ $B'$  majorizes } $B$ (written as $B \preceq B'$) if there exists $c\geq 1$  such that the estimate $(1)$ is fulfilled for every $f\in L(\mathbb{R}^n), r\in (0,\infty]$ and $\lambda\in (0,\infty)$.
\end{enumerate}

It is easy to see that $B_1\leq B_2 \leq B_3 \Rightarrow B_1 \leq B_3$ and $B_1\preceq B_2 \preceq B_3 \Rightarrow B_1 \preceq B_3$.

For a basis $B$ by $\overline{B}$ we will denote the family $\bigcup_{x\in \mathbb{R}^n} B(x)$. Following [2] let us call $\overline{B}$ the \emph{spread} of  $B$.

For a basis $B$ by $B_{\text{BF}}$ denote the basis defined as follows
$$
B_{\text{BF}}(x)=\{R\in \overline{B}: R \ni x\}\;\;\;\;\;\;(x\in\mathbb{R}^n).
$$
Let us call $B_{\text{BF}}$ the \emph{Busemann-Feller extension of a basis} $B$. It is easy to check that: 1) $B\subset B_{\text{BF}}$; 2) $B_{\text{BF}}$ is the smallest  Busemann-Feller basis containing $B$; 3) $B_{\text{BF}}$ is the largest basis having the same as $B$ spread;  4) If $B$ is homothecy invariant (translation invariant, convex) then $B_{\text{BF}}$ also is homothecy invariant (translation invariant, convex).

Note that the integral of a function $f\in L(\mathbb{R}^n)$ at a point $x\in \mathbb{R}^n$ may have the following seven different type limits of indeterminacy with respect to a basis $B$: i) $-\infty<\underline{D}_B(\int f,x)=\overline{D}_B(\int f,x)<\infty$; ii) $-\infty=\underline{D}_B(\int f,x)=\overline{D}_B(\int f,x)$; iii) $\underline{D}_B(\int f,x)=\overline{D}_B(\int f,x)=\infty$; iv) $-\infty=\underline{D}_B(\int f,x)<\overline{D}_B(\int f,x)=\infty$; v) $-\infty=\underline{D}_B(\int f,x)<\overline{D}_B(\int f,x)<\infty$; vi) $-\infty<\underline{D}_B(\int f,x)<\overline{D}_B(\int f,x)=\infty$; vii) $-\infty<\underline{D}_B(\int f,x)<\overline{D}_B(\int f,x)<\infty$.

\begin{remark}\label{r2}

If $B$ is a translation invariant convex basis then  the upper and lower derivatives with respect to $B$ for every function $f\in L(\mathbb{R}^n)$ possess the following "separation" property:
$$
 \underline{D}_B(\Int f,x)\leq f(x)\leq\overline{D}_B(\Int f,x)\;\;\text{for almost every}\;\;x\in\mathbb{R}^n.
$$
Indeed, there exists a translation invariant basis $B'\subset B$ for which $B'(0)=\{R_k\}$, where $R_1\supset R_2\supset\dots$ and $\diam R_k\rightarrow 0$. For the basis $B'$ it is valid Vitali type covering theorem (see [1, p.25]). It implies that $B'$ differentiates $L(\mathbb{R}^n)$. Consequently, for every $f\in L(\mathbb{R}^n)$, the estimations $ \underline{D}_B(\Int f,x)\leq\underline{D}_{B'}(\Int f,x)=f(x)=\overline{D}_{B'}(\Int f,x)\leq\overline{D}_B(\Int f,x)$ hold almost everywhere.
\end{remark}

\begin{remark}\label{r2}
In [3] (see Lemma 2) it was proved that if  $B$ is a density basis  then for every non-negative function $f\in L(\mathbb{R}^n)$ at almost every point $x\in\mathbb{R}^n$ it is valid equality $\underline{D}_{B}(\Int f,x)=f(x)$.
\end{remark}

\begin{remark}\label{r2}
From Remarks 1 and 2 it follows that if $B$ is a translation invariant convex density basis then  the integral of an arbitrary non-negative function $f\in L(\mathbb{R}^n)$ at almost every  point $x\in \mathbb{R}^n$ may have only the following three types of limits of indeterminacy with respect to the basis $B$: a) $\underline{D}_B(\int f,x)=f(x)=\overline{D}_B(\int f,x)<\infty$; b) $\underline{D}_B(\int f,x)=f(x)<\overline{D}_B(\int f,x)<\infty$; c) $\underline{D}_B(\int f,x)=f(x)<\overline{D}_B(\int f,x)=\infty$.

\end{remark}

\medskip

\textbf{2. Results}. Let we are given two bases $B$ and $B'$ with one and the same spread, i.e. $\overline{B}=\overline{B'}$. It seems natural the question weather differential properties of $B_1$ and $B_2$ are similar. In this regard it is true the following theorem showing that for the case of  translation invariant convex density bases the differential properties  are quite close.

\begin{theorem}
If $B$ is a  translation invariant convex density basis then its Busemann-Feller extension  $B_{\mathrm{BF}}$ differentiates the same class of non-negative functions as $B$, moreover, the integral of an arbitrary non-negative function $f\in L(\mathbb{R}^n)$ at almost every point $x\in \mathbb{R}^n$ has one and the same type limits of indeterminacy with respect to the bases $B$ and $B_{\mathrm{BF}}$.
\end{theorem}

Note that Theorem 1 gives some general extension principle  of results proved for Busemann-Feller bases to bases without restriction of being Busemann-Feller. Some applications of Theorem 1 will be given in the last section of the paper.
\begin{remark}\label{r2}
Theorem 1 is not true for an arbitrary density basis. Namely, by Hagelstein and Parissis in [4, see Proof of Theorem 3] it is constructed a translation invariant density basis $B$ in $\mathbb{R}$ the Busemann-Feller extension of which is not a density basis.
\end{remark}

\begin{remark}\label{r2}
Let $B$ be a basis in $\mathbb{R}^2$ for which $B(x)$ $(x\in\mathbb{R}^2)$ consists of all two-dimensional intervals of the type $[x_1,x_1+t_1]\times [x_2,x_2+t_2]$. Obviously, $B$ is homothecy invariant convex density basis and $B_{\mathrm{BF}}=\textbf{I}$. By Zerekidze [5] it was constructed a function $f\in L(\mathbb{R}^2)$ the integral of which is differentiable with respect to $B$ but is not differentiable with respect to $\textbf{I}$. Thus, the requirement of non-negativeness of functions in Theorem 1 is essential.
\end{remark}

Theorem 1 we obtain from the following two results.

\begin{theorem}
If $B$ and $B'$ are measurable density bases locally majorizing each other then the integral of an arbitrary non-negative function $f\in L(\mathbb{R}^n)$ at almost every point $x\in \mathbb{R}^n$ has one and the same type limits of indeterminacy with respect to the bases $B$ and $B'$.
\end{theorem}

\begin{theorem}
If $B$ is a  translation invariant convex density basis then its Busemann-Feller extension  $B_{\mathrm{BF}}$ is also density bases  and  $B$ and $B_{\mathrm{BF}}$ locally majorize each other.

\end{theorem}

\textbf{Acknowledgement}. The author was supported by Shota Rustaveli National Science Foundation (project no.  217282).

\medskip
\medskip

\textbf{3. Proof of Theorem 2.}
\begin{lemma}\label{l8}
Let $B$ and $B'$ be measurable bases and  $B \leq B'$. Then for an arbitrary non-negative function $f\in L(\mathbb{R}^n)$ and a  number $\lambda\in (0,\infty)$,
$$
\big|\big\{\overline{D}_{B}\left(\Int f, \cdot\right)\geq \lambda \big\}\setminus \big\{\overline{D}_{B'}\left(\Int f, \cdot\right)\geq \lambda/c \big\}\big|=0,
$$
where $c$ is the constant from $(1)$.
\end{lemma}

\begin{proof}
Passing to the limit in the estimation $(1)$ as $r \rightarrow 0$
we see for an arbitrary non-negative function $g \in L(\mathbb{R}^n)$ with bounded support that
$$
\big|\big\{\overline{D}_{B}\left(\Int g, \cdot\right)\geq \lambda \big\}\big|\leq c\big|\big\{\overline{D}_{B'}\left(\Int g, \cdot\right)\geq \lambda/c \big\}\big|.\eqno(2)
$$

Since the value of the upper derivative depends only on the local behaviour of the
function, we can conclude from $(2)$ that for an arbitrary non-negative function
$g \in L(\mathbb{R}^n)$  and an arbitrary cube $Q$
$$
\big|\big\{\overline{D}_{B}\left(\Int g, \cdot\right)\geq \lambda \big\}\cap Q\big|=
\big|\big\{\overline{D}_{B}\left(\Int g \chi_Q, \cdot\right)\geq \lambda \big\}\big|\leq
$$
$$
\leq c\big|\big\{\overline{D}_{B'}\left(\Int g \chi_Q, \cdot\right)\geq \lambda/c\big\}\big|=c\big|\big\{\overline{D}_{B'}\left(\Int g, \cdot\right)\geq \lambda/c\big\}\cap Q\big|. \eqno(3)
$$

Let us now assume that the converse to the assertion of the lemma holds, that is,
$$
\big|\big\{\overline{D}_{B}\left(\Int f, \cdot\right)\geq \lambda \big\}\setminus \big\{\overline{D}_{B'}\left(\Int f, \cdot\right)\geq \lambda/c \big\}\big|>0.
$$
Then there exists a point $x$ that is a density point of the set $\big\{\overline{D}_{B}\left(\Int f, \cdot\right)\geq \lambda \big\}$ and a point of rarefaction of the set $\big\{\overline{D}_{B'}\left(\Int f, \cdot\right)\geq \lambda/c \big\}$. Hence, considering a sufficiently small
cube $Q$ with centre at $x$, we obtain an inequality contradicting $(3)$. This proves
the lemma.
\end{proof}

\begin{lemma}\label{l8}
Let $B$ and $B'$ be measurable bases and  $B \leq B'$. Then for an arbitrary non-negative function $f\in L(\mathbb{R}^n)$,
$$
\big|\big\{\overline{D}_{B}\left(\Int f, \cdot\right)=\infty \big\}\setminus \big\{\overline{D}_{B'}\left(\Int f, \cdot\right)=\infty \big\}\big|=0.
$$
\end{lemma}

\begin{proof}
Passing to the limit in the estimation $(1)$, first as  $r \rightarrow 0$, and then as $\lambda\rightarrow \infty$ we see for an arbitrary nonnegative function $g \in L(\mathbb{R}^n)$ with bounded support that
$$
\big|\big\{\overline{D}_{B}\left(\Int g, \cdot\right)=\infty \big\}\big|\leq c\big|\big\{\overline{D}_{B'}\left(\Int g, \cdot\right)=\infty \big\}\big|.
$$
The rest part of the proof is analogous to the one given in Lemma 1.
\end{proof}

For a function $f\in L(\mathbb{R}^n)$ and $a>0$ denote $f_{[a]}=f\chi_{\{|f|\leq a\}}$ and $f^{[a]}=f\chi_{\{|f|>a\}}$.

\begin{lemma}\label{l8}
Let $B$  be a density  basis. Then for an arbitrary non-negative function $f\in L(\mathbb{R}^n)$ and a  number $\lambda\in (0,\infty)$,
$$
\bigg|\big\{\overline{D}_{B}\left(\Int f, \cdot\right)\geq f+ \lambda \big\}\bigtriangleup \bigcap_{k=1}^\infty\big\{\overline{D}_{B}\left(\Int f^{[k]}, \cdot\right)\geq \lambda \big\}\bigg|=0.
$$
\end{lemma}

\begin{proof}
It is known that each density basis differentiate the integral of every bounded summable function (see, e.g., [1, p.72]). Using this assertion for every $k\in\mathbb{N}$ we have that $B$ differentiates $\int f_{[k]}$. Consequently, it is easy to see that for every $k\in\mathbb{N}$,
$$
\big|\big\{\overline{D}_{B}\left(\Int f, \cdot\right)\geq f+ \lambda \big\}\bigtriangleup \big\{\overline{D}_{B}\left(\Int f^{[k]}, \cdot\right)\geq f^{[k]}+\lambda \big\}\big|=0.\eqno(4)
$$

For almost every point $x$ we have that $f^{[k]}(x)=0$ if $k$ is big enough. Hence it is easy to see that
$$
\bigg|\Big(\mathop{\underline{\lim}}\limits_{k\rightarrow\infty} \big\{\overline{D}_{B}\left(\Int f^{[k]}, \cdot\right)\geq f^{[k]}+\lambda \big\}\Big)\bigtriangleup \bigcap_{k=1}^\infty \big\{\overline{D}_{B}\left(\Int f^{[k]}, \cdot\right)\geq \lambda \big\}\bigg|=0.\eqno(5)
$$

From $(4)$ and $(5)$ we obtain the needed relation. The lemma is proved.
\end{proof}

\begin{lemma}\label{l8}
Let $B$ and $B'$ be measurable density bases and  $B \leq B'$. Then for an arbitrary non-negative function $f\in L(\mathbb{R}^n)$ and a number $\lambda\in (0,\infty)$,
$$
\big|\big\{\overline{D}_{B}\left(\Int f, \cdot\right)\geq f+ \lambda \big\}\setminus \big\{\overline{D}_{B'}\left(\Int f, \cdot\right)\geq f+ \lambda/c \big\}\big|=0,
$$
where $c$ is the constant from $(1)$.
\end{lemma}

\begin{proof}
By Lemma 3 we have
$$
\bigg|\big\{\overline{D}_{B}\left(\Int f, \cdot\right)\geq f+ \lambda \big\}\bigtriangleup \bigcap_{k=1}^\infty\big\{\overline{D}_{B}\left(\Int f^{[k]}, \cdot\right)\geq \lambda \big\}\bigg|=0,\eqno(6)
$$
$$
\bigg|\big\{\overline{D}_{B'}\left(\Int f, \cdot\right)\geq f+ \lambda/c \big\}\bigtriangleup \bigcap_{k=1}^\infty\big\{\overline{D}_{B'}\left(\Int f^{[k]}, \cdot\right)\geq \lambda/c \big\}\bigg|=0.\eqno(7)
$$
On the other hand by virtue of  Lemma 1 for every $k\in \mathbb{N}$ we have that
$$
\big|\big\{\overline{D}_{B}\left(\Int f^{[k]}, \cdot\right)\geq \lambda \big\}\setminus \big\{\overline{D}_{B'}\left(\Int f^{[k]}, \cdot\right)\geq \lambda/c \big\}\big|=0.\eqno(8)
$$

From $(6)-(8)$ we easily obtain the needed relation. The lemma is proved.

\end{proof}

\begin{lemma}\label{l8}
Let $B$ and $B'$ be measurable density bases and  $B \leq B'$. Then for an arbitrary non-negative function $f\in L(\mathbb{R}^n)$,
$$
\big|\big\{\overline{D}_{B}\left(\Int f, \cdot\right)> f \big\}\setminus \big\{\overline{D}_{B'}\left(\Int f, \cdot\right)> f \big\}\big|=0.
$$
\end{lemma}

\begin{proof}
We have that $\big\{\overline{D}_{B}\left(\Int f, \cdot\right)> f \big\}=\bigcup_{m=1}^\infty\big\{\overline{D}_{B}\left(\Int f, \cdot\right) \geq f +1/m \big\},$ and $\big\{\overline{D}_{B'}\left(\Int f, \cdot\right)> f \big\}=\bigcup_{m=1}^\infty\big\{\overline{D}_{B'}\left(\Int f, \cdot\right) \geq f +1/(cm) \big\}.$ Here $c$ is the constant from $(1)$. Now using Lemma 4 for  function $f$ and for every $\lambda=1/m$ $(m\in\mathbb{N})$ we obtain the needed relation. The lemma is proved.
\end{proof}

Taking into account Remark 3   from Lemmas 2 and 5 we obtain Theorem 2.

\medskip
\medskip

\textbf{4. Proof of Theorem 3.}  For a basis $B$ and a non-degenerate linear mapping $M:\mathbb{R}^n\rightarrow \mathbb{R}^n$ by $B_M$ denote the basis for which
$$
B_M(M(x))=\{M(R): R\in B(x)\}\;\;\;\;\;\;\;(x\in \mathbb{R}^n).
$$

\begin{lemma}\label{l8}
If  $B$ is   a density basis then for every non-degenerate linear mapping $M:\mathbb{R}^n\rightarrow \mathbb{R}^n$ the basis $B_M$ also possesses the density property.

\end{lemma}

\begin{proof}
Observe that for every bounded measurable set $E$, $x\in\mathbb{R}^n$ and $R\in B(x)$,
$$
\frac{|M(R)\cap M(E)|}{|M(R)|}=\frac{|M(R\cap E)|}{|M(R)|}=\frac{|J_M||R\cap E|}{|J_M||R|}= \frac{|R\cap E|}{|R|}.
$$
Here $J_M$ denotes the Jacobian of the mapping $M$.  Thus, the integral means with respect to $B$ of the function $\chi_{E}$ at the point $x$ are the same as the integral means with respect to $B_M$ of the function $\chi_{M(E)}$ at the point $M(x)$. It easily implies the differentiability of  $\int\chi_{M(E)}$ with respect to $B_M$. Since $E$ is an arbitrary bounded measurable set we conclude that $B_M$ is a density basis. The lemma is proved.
\end{proof}

We will  need the following characterization of translation invariant density bases proved by Hagelstein and Parissis [4].

\medskip

\textbf{Theorem A.} \emph{Let $B$ be a translation invariant basis. Then the following properties are equivalent:} 1) $B$ \emph{is a density basis;} 2) \emph{For each $\lambda\in (0,1)$ there exist  positive constants $r(B,\lambda)$ and  $c(B,\lambda)$ such that for each measurable set $E$ one has: $|\{M_B^{r(B,\lambda)}(\chi_E)\geq \lambda\}|\leq c(B,\lambda)|E|$.}

\medskip

Note that Theorem A for  Busemann-Feller translation invariant bases was proved in  [6].

\smallskip

Below without loss of generality assume that for a convex basis $B$ sets forming $B$ (i.e. sets from the spread $\overline{B}$)  are closed.

For a set $A $ with the centre of symmetry at a point $x$ and for a number $\alpha>0$
we denote by $\alpha A$ the dilation of $A$ with coefficient $\alpha$: $\alpha A = \{x+\alpha(y-x) : y \in A\}$.

 By F. John it was proved that (see, e.g., [1, p.139])  for any    bounded closed convex set $E$ in $\mathbb{R}^n$ with positive measure there exists a closed  ellipsoid $T$ for which $T\subset E\subset nT$. This assertion easily implies the following lemma (see, [3, Lemma 3] for details).

\begin{lemma}\label{l8}
For any    bounded closed convex set $E$ in $\mathbb{R}^n$ with positive measure there exists a closed $n$-dimensional rectangle $R$ such that $R\subset E\subset n^2 R$.

\end{lemma}

For every set $E$ from $B(0)$ using Lemma 7 we can find  a closed rectangles $R_\ast(E)$ and $R^\ast(E)$ such that $R^\ast(E)=n^2R_\ast(E)$ and  $R_\ast(E)\subset E\subset R^\ast(E)$.

Let $B_\ast$ and $B^\ast$ be the translation invariant bases the values of which at the origin are defined as  follows:
$$
B_\ast(0)=\{\{0\}\cup R_\ast(E): E\in B(0)\},
$$
$$
B^\ast(0)=\{R^\ast(E): E\in B(0)\}.
$$
Note that the sets forming $B_\ast$ in general are not rectangles. The reason of it is that  for a set  $E\in B(0)$  the rectangle $R_\ast(E)$ may not  contain the origin. It is easy to check the that there are valid the majorizations: $B_\ast \leq B \leq B_{\mathrm{BF}} \leq B_{\mathrm{BF}}^\ast$, moreover, for every $f\in L(\mathbb{R}^n)$, $r\in (0,\infty]$ and $x\in\mathbb{R}^n$ there are valid the pointwise estimations: $M_{B_\ast}^r(f)(x)\leq n^{2n} M_{B}^{n^2 r}(f)(x),$ $M_{B}^{r}(f)(x)\leq M_{B_\mathrm{BF}}^{r}(f)(x),$ $M_{B_\mathrm{BF}}^{r}(f)(x)\leq n^{2n} M_{B_{\mathrm{BF}}^\ast}^{n^2 r}(f)(x).$

For a basis $B$ and a positive number $\alpha$ by $c B$ denote the basis $B_M$ where $M$ is the homothecy with centre at the origin and the coefficient equal to $c$.

For a basis $B$ by $B_{\mathrm{sym}}$ denote the basis $B_M$ where $M$ is the symmetry with respect to the origin.

\begin{lemma}\label{l8}
 For every $f\in L(\mathbb{R}^n)$ and $r \in (0,\infty]$,
$$
\{M_{B_{\mathrm{BF}}^\ast}^r (f)\geq \lambda\}\subset \{M_{(c_2B_\ast)_{\mathrm{sym}}}^{c_1 r} (\chi_{\{M_{c_4 B_\ast}^{c_3r}(f)\geq \lambda/c_5\}})\geq 1/c_6 \},
$$
where $c_1,\dots ,c_6$ are positive constants depending only on $n$.
\end{lemma}

\begin{proof}
Suppose $M_{B_{\mathrm{BF}}^\ast}^r (f)(x)\geq \lambda$. Let  $R^\ast \in B_{\mathrm{BF}}^\ast(x)$ be a rectangle for which $\diam R^\ast<r$ and $\int_{R^\ast} |f|\geq\lambda|R^\ast|/2$. Denote $R_\ast=\frac{1}{n^2}R^\ast$. It is easy to see that there is a point $x_\ast\in R^\ast$  for which $\{x_\ast\}\cup R_\ast\in B_\ast(x_\ast)$. Without loss of generality assume that $R^\ast$ is an interval of the type $[0,t_1]\times \dots \times [0,t_n]$ and the point $x_\ast$ is to the "left" of the centre of $R_\ast$, i.e. $x_1\leq t_1/2,\dots, x_n\leq t_n/2$. Let us consider the interval $R$ with the centre at the origin which is the translate of $4R^\ast$. It is easy to see that there is a point $y$ lying to left form the origin for which $\{y\}\cup R\in (4n^2 B_\ast)(y)$. Obviously, $y$ belongs to the rectangle $n^2R$. Now let us consider the set of all points $y+t$ with the properties: $t$ lies to the left from the origin and $R+t\supset R^\ast$. It is easy to check that $P$ is the translate of the rectangle $R^\ast$. Let us consider an arbitrary point $y+t$ from $P$. Since $\{y\}\cup R\in (4n^2 B_\ast)(y)$  then $\{y+t\}\cup (R+t)\in (4n^2 B_\ast)(y+t)$. Consequently, taking into account that $\diam(\{y+t\}\cup (R+t))= \diam(\{y\}\cup R)\leq \diam (n^2R)=4n^2\diam R< 4n^2r$ we have
$$
M_{4n^2 B_\ast}^{4n^2r}(f)(y+t)\geq \frac{1}{|R|}\int_{R+t}|f|\geq \frac{1}{4^n|R^\ast|}\int_{R^\ast}|f|\geq\frac{\lambda}{2 \cdot4^n}.
$$
Thus,
$$
P\subset\{M_{4n^2 B_\ast}^{4n^2r}(f)\geq\lambda/(2 \cdot4^n)\}.\eqno(9)
$$
Let $z_0$ be the centre of symmetry of the union of the rectangles $R^\ast$ and $P$. By $S$ denote the symmetry with respect to $z_0$.  Clearly, $S(x)\in P$. Then $\{S(x)\}\cup (R+s(x)-y)\in (4n^2 B_\ast)(S(x))$ and $(R+S(x)-y)\supset R^\ast$. Consequently, taking into account the definition of the basis $(4n^2 B_\ast)_{\mathrm{sym}}$ we have that $\{x\}\cup (S(R)+x-S(y))\in (4n^2 B_\ast)_{\mathrm{sym}}(x)$ and $(S(R)+x-S(y))\supset P$. Wherefrom we have
$$
M_{(4n^2 B_\ast)_{\mathrm{sym}}}^{4n^2r}(\chi_P)(x)\geq \frac{1}{|S(R)|}\int_{S(R)+x-S(y)}\chi_P=\frac{1}{4^n|R^\ast|}|P|=\frac{1}{4^n}.\eqno(10)
$$
From $(9)$ and $(10)$ we conclude the validity of the lemma.

\end{proof}

\begin{lemma}\label{l8}
 For every $f\in L(\mathbb{R}^n)$ and $r \in (0,\infty]$,
$$
\{M_{B_{\mathrm{BF}}^\ast}^r (f)\geq \lambda\}\subset \{M_{3 B_{\mathrm{BF}}^\ast}^{3r} (\chi_{\{M_{B_\ast}^{r}(f)\geq \lambda/\alpha_1\}})\geq 1/\alpha_2 \},
$$
where $\alpha_1$ and $\alpha_2$ are positive constants depending only on $n$.
\end{lemma}

\begin{proof}
Suppose $M_{B_{\mathrm{BF}}^\ast}^r (f)(x)\geq \lambda$. Let  $R^\ast \in B_{\mathrm{BF}}^\ast(x)$ be a rectangle for which $\diam R^\ast<r$ and $\int_{R^\ast} |f|\geq\lambda|R^\ast|/2$.  Denote $R_\ast=\frac{1}{n^2}R^\ast$. It is easy to see that there is a point $x_\ast\in R^\ast$  for which $\{x_\ast\}\cup R_\ast\in B_\ast(x_\ast)$. Let us decompose  $R^\ast$  into rectangles which are translates of $\frac{1}{2}R_\ast$.  One of this rectangles obviously will satisfy the estimation: $\int_I |f|\geq\lambda|I|/2$. Denote
$$
P=\{x_\ast+t: t\in \mathbb{R}^n, R_\ast+t\supset I\}.
$$
It is easy to see that $P$ is the translate of $I$ and $P\subset 3R^\ast$. Let us consider an arbitrary point $x_\ast+t$ from $P$. Since $\{x_\ast\}\cup R_\ast\in B(x_\ast)$  then $\{x_\ast+t\}\cup (R_\ast+t)\in B(x_\ast+t)$. Consequently, taking into account that $\diam(\{x_\ast+t\}\cup (R_\ast+t))=\diam(\{x_\ast\}\cup R_\ast)\leq \diam R^\ast<r$ we have
$$
M_{B_\ast}^r(f)(x_\ast+t)\geq \frac{1}{|R_\ast|}\int_{R_\ast+t}|f|\geq \frac{1}{2^n|I|}\int_I|f|\geq\frac{\lambda}{2^{n+1}}.
$$
Thus,
$$
P\subset\{M_{B_\ast}^r(f)\geq\lambda/2^{n+1}\}. \eqno(11)
$$
By  the inclusion  $P\subset 3R^\ast$ we write
$$
M_{3B_{\mathrm{BF}}^\ast}^{3 r} (\chi_P)(x)\geq \frac{1}{|3R^\ast|}\int_{3R^\ast}\chi_P=\frac{|P|}{6^n n^{2n}|I|}=\frac{1}{6^n n^{2n}}. \eqno(12)
$$
From $(11)$ and $(12)$ we conclude the validity of the lemma.
\end{proof}

\medskip

Now let us move directly to the proof of Theorem 3.  Since $B$ is a density basis then by virtue of majorisation $B_\ast\leq B$ and Theorem A we conclude that $B_\ast$ also is a density basis. Further using Lemma 6 by Theorem A we have that $c_4B_\ast$ and $(c_2B_\ast)_{\mathrm{sym}}$ are density basis. Consequently, using Lemma 8 and Theorem A it is easy to see that  $B_{\mathrm{BF}}^\ast$ is a density basis. Hence by virtue of Lemma 9 and Theorem A we conclude the validity of the majorization $B_{\mathrm{BF}}^\ast\leq B_\ast$. Now, taking into account the relations $B_\ast \leq B \leq B_{\mathrm{BF}} \leq B^\ast_{\mathrm{BF}}$ and using Theorem A we conclude that $B_{\mathrm{BF}}$ is a density basis and  the bases $B_\ast, B, B_{\mathrm{BF}}$ and $B^\ast_{\mathrm{BF}}$ locally majorize each other. The theorem is proved.

\begin{remark}\label{r2}
Taking into account the properties of the basis $B^\ast$ from the proof of Theorem 3 we  obtain the following result: For every translation invariant convex density basis $B$ there exists a translation  invariant Busemann-Feller basis $B^\ast$ formed by $n$-dimensional rectangles such that the integral of an arbitrary non-negative function $f\in L(\mathbb{R}^n)$ at almost every point $x\in \mathbb{R}^n$ have one and the same type limits of indeterminacy with respect to the bases $B$ and $B^\ast$.
\end{remark}

\begin{remark}\label{r2}
Let us recall the following well-known result (see, e.g., [1,p.77]): If  $B$ is  a density basis, $f\in  L(\mathbb{R}^n)$ and $B$ differentiates $\Int |f|$ then $B$ differentiates $\Int f$ also. Taking into account this result from Theorem 1 we obtain the following assertion: Let  $B$ be a  translation invariant convex density basis and $\varphi(L)(\mathbb{R}^n)$ be a some integral class. Then $B$ differentiates $\varphi(L)(\mathbb{R}^n)$ if and only if   $B_{\mathrm{BF}}$ differentiates $\varphi(L)(\mathbb{R}^n)$.

\end{remark}

\begin{remark}\label{r2}

It is true  the following characterization of homothecy invariant density basis (see [1, p.69]).

\medskip

\textbf{Theorem B}. Let $B$ be a homothecy invariant basis. Then the following properties are equivalent: 1)  $B$  a density basis; 2)  For each $\lambda\in (0,1)$ there exists a positive constant $c(B,\lambda)$ such that for each bounded measurable set $E$ one has: $|\{M_B(\chi_E)\geq \lambda\}|\leq c(B,\lambda)|E|$.

\medskip

Using the above result instead of Theorem A it can be proved that if a basis $B$ in Theorem 3 is homothecy invariant then $B$ and $B_{\mathrm{BF}}$  majorize each other.

\end{remark}

\medskip
\medskip
\textbf{4. Some applications.}

\smallskip

\textbf{I}. Besikovitch [7] proved the following theorem about the limits of indeterminacy for the strong differentiation process: For an arbitrary  function $f\in L(\mathbb{R}^2)$ at almost every  point $x\in \mathbb{R}^2$ there are valid the implications:
$$
\underline{D}_\textbf{I}(\Int f,x)\neq f(x) \;\Rightarrow \;\underline{D}_\textbf{I}(\Int f,x)=-\infty,
$$
$$
\overline{D}_\textbf{I}(\Int f,x)\neq f(x) \;\Rightarrow \;\overline{D}_\textbf{I}(\Int f,x)=\infty.
$$

An analog of this result for the multi-dimensional case was obtained by Ward [8]. Note that the multi-dimensional extension also may be obtained using the version of F. Riesz "rising sun" lemma proved in [9].

Guzm\'an [10, p.389] posed the following problem: \emph{To what bases can Besicovitch's result be extended}?

We say that a basis $B$ possesses the \emph{Besicovitch property}  (the \emph{weak Besicovitch property}) if for an arbitrary function $f\in L(\mathbb{R}^n)$ (for an arbitrary non-negative function $f\in L(\mathbb{R}^n)$) at almost every  point $x\in \mathbb{R}^n$ there are valid the implications:
$$
\underline{D}_B(\Int f,x)\neq f(x) \;\Rightarrow \;\underline{D}_B(\Int f,x)=-\infty,
$$
$$
\overline{D}_B(\Int f,x)\neq f(x) \;\Rightarrow \;\overline{D}_B(\Int f,x)=\infty.
$$

\begin{remark}\label{r2}
 If a  $B$ is not a density basis then it is easy to see that there is a measurable bounded set $E$ for which the set $\{x\in E:\underline{D}_B(\Int \chi_E,x)<1\}$ is not of measure zero. Consequently, if a basis $B$ possesses the weak Besikovitch property then $B$ is a  density basis.

\end{remark}

It is valid the following characterization of homothecy invariant convex bases possessing the weak Besicovitch property.

\begin{theorem}
A homothecy invariant  convex basis $B$ possesses the weak Besicovitch property if and only if $B$ is  a density basis.
\end{theorem}

Theorem 4 we obtain using Theorem 1 and the next result proved in [3]: A Busemann-Feller homothecy invariant  convex basis $B$  possesses the weak Besicovitch property if and only if $B$ is  a density basis.

Note that the result analogous to Theorem 1 for Busemann-Feller homothecy invariant convex basis formed by central-symmetric sets was proved by Guzm\'an and Men\'argues  [1, p.106].

\begin{remark}\label{r2}
The analogue of Theorem 1 is not valid for translation invariant convex bases. Namely, in [11] it was constructed an example of a Busemann-Feller translation invariant basis formed by two-dimensional intervals which
does not possess the weak Besicovitch property.

\end{remark}

\begin{remark}\label{r2}
It is unknown whether it is valid the characterization of homothecy invariant convex bases possessing the Besicovitch property analogous to Theorem 1.   Moreover, the question  is open even for homothecy invariant bases
formed by $n$-dimensional intervals. Some partial result in this direction is obtained in  [12].
\end{remark}

\medskip

\textbf{II}. A basis $B$ is called \emph{regular} if there is a number $c\geq 1$ such that for every set $R$ from the spread of the basis $B$ there exists a cubic interval $Q$ with the properties: $R\subset Q$ and $|Q|\leq c |R|$. Note that (see, e.g. [1, p.25]) each regular basis differentiates the class $L(\mathbb{R}^n)$.

Let $B$ be a homothecy invariant density basis. The \emph{halo function} $\varphi_B$ of the basis $B$ is defined as follows: $\varphi_B(t)=1$ if $0\leq t\leq 1$ and
$$
\varphi_B(t)=\sup_E\frac{|\{M_B(\chi_E)\geq 1/t\}|}{|E|}\;\;\;\text{if}\;\;\;t>1,
$$
where the supremum is taken over all bounded measurable sets $E$  with positive measure.

The \emph{halo conjecture} (see, [1, p. 178]) asserts that each  homothecy invariant density basis  $B$ differentiates the integral class $\varphi_B(L)(\mathbb{R}^n)$.

By Moriy\'on (see, [1,p.206]) it was characterised Busemann-Feller homothecy invariant density bases formed by open sets which differentiate the class $L(\mathbb{R}^n)$ as well as halo conjecture was justified for  bases with halo function satisfying the condition $c_1 t\leq \varphi_B(t)\leq c_2 t$ $(t\geq0)$. The result  is:

Let $B$ be a Busemann-Feller homothecy invariant density basis formed by open sets. Then the following statements are equivalent:
\begin{enumerate}
\item[$\bullet$] $B$ differentiates the class $L(\mathbb{R}^n)$;

\item[$\bullet$] $B$ is a regular basis;

\item[$\bullet$] The halo function of $B$ satisfies the condition: $c_1 t\leq \varphi_B(t)\leq c_2 t$ $(t\geq0)$, where $c_1$ and $c_2$ are positive constants.

\end{enumerate}

Taking into account Theorem B and Remark 7 from Theorems 1 and 3 we can easily conclude that the analog  of Moriy\'on's result  is true for every homothety invariant convex density basis.

\medskip

\textbf{III}. Hagelstein and Stokolos [13] proved that  every Busemann-Feller homothecy invariant  convex density basis differentiates the class $L^p(\mathbb{R}^n)\cap L(\mathbb{R}^n)$ for sufficiently large $p$. Taking into account Remark 7 and using Theorem 1  this result can be extended to  every   homothecy invariant  convex density basis.

\end{document}